\documentclass[a4paper,12pt]{amsart}
\usepackage[colorlinks=true, linkcolor=blue, citecolor=blue]{hyperref}
\usepackage{amscd,amssymb,epsfig,amsxtra,amsfonts}
\usepackage{amsmath}
\usepackage{stmaryrd}
\usepackage{mathrsfs}
\usepackage{upgreek}
\usepackage{marginnote}
\usepackage{verbatim}
\usepackage[utf8]{inputenc}
\usepackage{diagrams}
\usepackage{epsfig}
\newcommand{\ms}[1]{\mbox{\tiny$#1$}}
\newcommand{\epsh}[2]
{\begin{array}{c} \hspace{-1.3mm}
   \raisebox{-4pt}{\epsfig{figure=#1,height=#2}}
   \hspace{-1.9mm}\end{array}}

\newcommand{\slt}{{\mathfrak{sl}_2}}

\newcommand{\ith}{{\textsuperscript{th}} }

\newcommand{\ve}{\varepsilon}

\newcommand{\e}{{\operatorname{e}}}
\newcommand{\Id}{\operatorname{Id}}
\renewcommand{\Im}{\operatorname{Im}}

\newcommand{\End}{\operatorname{End}}
\newcommand{\Hom}{\operatorname{Hom}}
\newcommand{\Proj}{\operatorname{\mathsf{Proj}}}

\newcommand{\tr}{\operatorname{tr}}
\newcommand{\ptr}{\operatorname{ptr}}

\newcommand{\C}{\ensuremath{\mathbb{C}}}
\newcommand{\Z}{\ensuremath{\mathbb{Z}}}
\newcommand{\R}{\ensuremath{\mathbb{R}}}

\newcommand{\K}{\ensuremath{\mathbb{K}}}

\newcommand{\wh}{\widehat}
\newcommand{\wb}{\overline}

\newcommand{\si}{{\mu}}

\newcommand{\bp}[1]{{\left(#1\right)}}
\newcommand{\bcr}[1]{{\left\{#1\right\}}}
\newcommand{\qn}[1]{{\left\{#1\right\}}}

\newcommand{\cU}{\ensuremath{\mathcal{U}}}

\newcommand{\HH}{{\mathcal H}}
\newcommand{\RR}{{\mathcal R}}

\newcommand{\UK}{\cU^K}
\newcommand{\Uo}{\ensuremath{{\widetilde\U}}}
\newcommand{\sio}{\widetilde\si}

\newcommand{\cat}{{\operatorname{\mathscr{C}}}}

\newcommand{\mt}{\operatorname{\mathsf t}}
 
\newcommand{\Ho}{{\mathbf{H}}}
\newcommand{\U}{{\mathbf{U}}}
\newcommand{\HV}{{\mathrm{HV}}}
\newtheorem{theo}{Theorem}[section]
\newtheorem{Lem}[theo]{Lemma}
\newtheorem{Prop}[theo]{Proposition}

\newcounter{axi}

\theoremstyle{definition}
\newtheorem{Def}[theo]{Definition}
\newtheorem{rmq}[theo]{Remark}

\newcounter{exo} \newcounter{numexercice}
\renewcommand{\theexo}{\arabic{exo}}

\begin{document}
\title{Modified symmetrized integral in $G$-coalgebras}

\author{Nathan Geer}
\address{Utah State University, Department of Mathematics and Statistics, Logan UT 84341, USA}
\email{nathan.geer@usu.edu}
\author{Ngoc Phu HA}        
\address{Hung Vuong University, Faculty of Natural Sciences, Viet Tri, Phu Tho, Viet Nam}
\email{ngocphu.ha@hvu.edu.vn}
\author{Bertrand Patureau-Mirand}
\address{Universit\'e Bretagne Sud, Laboratoire de Math\'ematiques de Bretagne Atlantique, UMR CNRS 6205, Vannes, France}
\email{bertrand.patureau@univ-ubs.fr}

\maketitle

\begin{abstract}
  For $G$ a commutative group, we give a purely Hopf $G$-coalgebra
  construction of $G$-colored $3$-manifolds invariants using the
  notion of modified integral.
\end{abstract}

\vspace{25pt}



\section{Introduction}
In \cite{Oh93}, Ohtsuki  defined the structure of a ribbon colored Hopf algebra and
he reveal such a structure based on 
a non-semisimple example coming from  
quantum $\slt$. 
He showed that a ribbon colored Hopf algebra leads to a universal link invariant taking values in some quotient of the colored Hopf algebra.  In the example of quantum $\slt$ this universal invariant recovers the ADO link invariant given in \cite{ADO}.  

A non commutative version of the notion of a ribbon colored Hopf algebra 
called a $G$-coalgebra was introduced by Turaev 
with the goal of producing Homotopy Quantum Field Theories (HQFTs),
see \cite{Turaev10}.  Virelizier showed that certain $G$-coalgebras gives rise to
universal invariants of a $G$-link (i.e.\  a link with a flat
connection in a principal $G$-bundle over the complement of the link).
Moreover, he shows that if the $G$-coalgebra has a $G$-trace this
universal invariant leads to a \emph{Hennings-Virelizier invariant} of
3 dimensional $G$-manifolds (3-manifolds endowed with a flat
$G$-bundle), see \cite{Vire01} and \cite[Appendix 7.2]{Turaev10}.
 Moreover, in
\cite[Theorem 7.4]{Vire02}, Virelizier shows that a $G$-trace exists
when the $G$-coalgebra has certain conditions (for example, when it is
semisimple and finite type).  Constructing non-trivial $G$-manifold
invariants in the non-semisimple and non-finite type setting seems to require new
techniques which we propose here: we
show that in a special case (when $G$ is commutative and the
$G$-coalgebra has certain conditions including a modified symmetrized integral)
Virelizier's construction can be re-normalized and leads to
non-trivial $G$-manifold invariants.  
In particular, we show that unrolled quantum group associated to $\slt$ gives
rise to $G$-coalgebra with modified symmetrized
integral and so an invariant of $G$-manifolds. The link invariant underlying this example is closely related to the universal link invariant defined by Ohtsuki in \cite{Oh93}. This link invariant was previously used to produce a representation theoretic Reshetikhin-Turaev type 3-manifold invariants in \cite{FcNgBp14}. 
Here we give an Hopf $G$-coalgebra approach to make a Hennings type 3-manifold invariants.  

The main new concept of this paper is the notion of a modified symmetrized integral on a $G$-coalgebra.  This notion is related and inspired by the theory of modified traces and symmetrized integrals given in \cite{BBG,GKP2011, GPV2013}. 
We show that modified symmetrized integrals exist in the general context of Theorem \ref{T:ModSymIntExist}.  This general context is inspired and modeled on topological unrolled quantum groups associated to simple Lie algebras, see \cite{GHP20}.     
However, the work of \cite{GHP20} requires technical topological completions where this paper uses more straightforward algebraic techniques.  

The outline and main results of the paper are described as follows.  
In Section \ref{S:VirelizierInv}, when $G$ is commutative, we 
describe 
Virelizier version of the universal invariant of $G$-links and its corresponding 3-dimensional $G$-manifold invariant which we call the graded Hennings-Virelizier invariant.    In this context, a $G$-manifolds $(M,\omega)$ is
just the data of a 3-manifold with a cohomology class $\omega\in H^1(M,G)$. 
In Section \ref{S:ModifiedIntInv} we give the notion of a (modified) symmetrized integral and show that such an integral can be used to renormalize the graded Hennings-Virelizier invariant.  In Section \ref{s:ex}  we show that Ohtsuki's ribbon colored Hopf algebra associated to $\slt$ leads to an example of the invariants defined in this paper.

\subsection*{Acknowledgments}
N.G.\ is partially
supported by NSF grants DMS-1664387 and DMS-2104497.  He would also like to thank the Max Planck Institute for Mathematics in Bonn for its hospitality during work on this paper. N. P. H.\ would like to thank the fundamental research program of Hung Vuong University. 

\section{Virelizier invariant of $G$-manifolds: the abelian case}\label{S:VirelizierInv}
Let $ \K$ be an integral domain.

\subsection{Ribbon Hopf $G$-coalgebra over abelian group}
The notion of a Ribbon Hopf $G$-coalgebra is a tool used to produce homotopy R-matrices and examples of non commutative HQFTs.
In this section we give a simplified version of a Hopf $G$-coalgebra in
the quasi-triangular case which coincide with Ohtsuki's notion of
colored Hopf algebra \cite{Oh93}.  In particular the group $G$ is
commutative and we use additive notation.
 
A Hopf $G$-coalgebra is a family $\Ho = \{\Ho_a\}_{a\in G}$ of
$\K$-algebras (with a product $m_a:\Ho_{a} \otimes
\Ho_a\to\Ho_a$ and a unit $\eta_a:\K\to\Ho_{a}$ for each $a\in G$) endowed with a
comultiplication
$\Delta = \{\Delta_{a,b} : \Ho_{a+b} \to \Ho_{a} \otimes
\Ho_b\}_{a,b\in G}$, a counit $\ve: \Ho_0 \to \K$ and  an
antipode $S = \{S_a : \Ho_a \to \Ho_{-a}\}_{a\in G}$ which satisfy
(see \cite{Oh93, Turaev10,Vire02}):
\begin{enumerate}
\item $\Delta$ is 
coassociative: for any $a,b,c\in G$,
  $$\left(\Delta_{a,b}\otimes \Id_{\Ho_c}\right)\circ \Delta_{a+b,c}=
  \bp{ \Id_{\Ho_a}\otimes \Delta_{b,c}}\circ
  \Delta_{a,b+c}, $$
\item $\ve$ is a counit:
  $\bp{ \Id_{\Ho_a}\otimes \ve}\circ \Delta_{a,0}=\bp{\ve\otimes
    \Id_{\Ho_a}}\circ \Delta_{0,a}=\Id_{\Ho_a}$,
\item $S$ is a family of 
bijective antipode:
  for any $a\in G$,
  $$m_a\circ\bp{S_{-a}\otimes \Id_{\Ho_a}}\circ
  \Delta_{-a,a}=\eta_a\circ \ve=m_a\circ\bp{\Id_{\Ho_{a}}\otimes
    S_{-a}}\circ \Delta_{a,-a},$$
\item $\Delta$ and $\ve$ are algebra maps (for the standard algebra
  structure of $\Ho_a\otimes\Ho_b$ and $\K$).
\end{enumerate}

Let us now recall some definitions from \cite{Vire02,Ha18}.
\begin{Def}[pivotal structure]
  Let $\Ho=\{\Ho_{a}\}_{a\in G}$ be a Hopf $G$-coalgebra.
  \begin{enumerate}
  \item A \emph{$G$-grouplike} element is a family
    $\{x_{a}\in \Ho_{a}\}_{a\in G}$ such that
    $\Delta_{a,b}(x_{a+b})=x_{a}\otimes x_{b}$, for all
    $a,b\in G$ and $\ve(x_0)=1$.
  \item A \emph{pivot} for $\Ho$ is a $G$-grouplike element
    $\{g_{a}\}_{a\in G}$ such that, 
    for all $ a\in G$ for any $x\in \Ho_a$, 
    $S_{-a}S_a(x)=g_axg_a^{-1}$.  If 
    $\Ho$ has a pivot we say it is a \emph{pivotal} Hopf
    $G$-coalgebra.
   \end{enumerate}
\end{Def}

\begin{Def}[quasitriangular structure]\label{Def:quasitriangular} 
Let  $\tau: \Ho_{a}\otimes \Ho_{b}\to \Ho_{b}\otimes \Ho_{a}$ be the linear map
 defined by $x\otimes y \mapsto y\otimes x$.  A Hopf $G$-coalgebra
  $\Ho=\{\Ho_{a}\}_{a\in G}$ is quasitriangular if it has an R-matrix,
  which is a family of invertible elements
  $\RR=\{\RR_{a, b}\in \Ho_a\otimes \Ho_b\}_{a,b\in G}$ satisfying for
  any $a,b,c\in G$
  \begin{enumerate}
  \item\label{E:Rtau}$\RR_{a,b}.\Delta_{a,b}(x)=\tau\bp{\Delta_{b,a}(x)}.\RR_{a,b}$
    for any $x\in \Ho_{a+b}$,
  \item \label{E:DeltaR12}
    $\bp{\Id_{\Ho_a}\otimes
      \Delta_{b,c}}\RR_{a,b+c}=\bp{\RR_{a,c}}_{1b3}\bp{\RR_{a,b}}_{12c}$,
  \item \label{E:DeltaR23}
    $\bp{\Delta_{a,b}\otimes
      \Id_{\Ho_c}}\RR_{a+b,c}=\bp{\RR_{a,c}}_{1b3}\bp{\RR_{b,c}}_{a23}$
  \end{enumerate}
  where
$\bp{\RR_{a,c}}_{1b3}=r_a\otimes 1_b\otimes s_c,\  \bp{\RR_{a,b}}_{12c}=r_a\otimes s_b\otimes 1_c $ and $\bp{\RR_{b,c}}_{a23}=1_a\otimes r_b\otimes s_c$ uses the notation  $r_a\otimes s_b$ (with an implied summation) for the element $\RR_{a,b}\in \Ho_a\otimes \Ho_b$.
\end{Def}

In \cite{Vire02},  Virelizier defines a ribbon Hopf $G$-coalgebra using a ribbon element (a similar style of definition is given in \cite{Oh93}).  Here we give an equivalent definition in terms of the pivotal structure and a twist element.  
\begin{Def}[ribbon structure]\label{D:RibbonStructure}
  A pivotal quasitriangular Hopf $G$-coalgebra is \emph{ribbon} if for any
  $a\in G$ the element 
  $\theta_a=m_a(\tau((g_a\otimes1_a)R_{a,a}))$
  satisfies
    \begin{equation}
      \label{E:twist}
      \theta_a=m_a\bp{(1_a\otimes g_{a}^{-1})R_{a,a}}.
    \end{equation}
\end{Def}
\begin{Def}[symmetrized integral]\label{D:si}
  A \emph{symmetrized integral} for a pivotal Hopf $G$-coalgebra is a family
  of linear forms $\si=\{\si_a\in \Ho_{a}^*\}_{a\in G}$ such that
  \begin{align}
    \bp{\si_{a}\otimes g_{b}}\Delta_{a,b}(x)
    &=\si_{a+b}(x)1_{b} \text{ for }
      x\in\Ho_{a+b},\label{eq:rightsi}\\
    \si_{a}(xy)&=\si_{a}(yx)\ \text{for}\ x, y\in \Ho_{a},\label{eq:mucyc}\\
    \si_{-a}\bp{S_{a}(x)}&=\si_{a}(x)\ \text{for}\
                           x\in\Ho_{a}.\label{eq:muS}
  \end{align}
  A symmetrized integral  is {\em twist non
    degenerate} if the scalar
  $$\mu_0(g_0\theta_0)\mu_0(g_0^{-1}\theta_0^{-1})$$
is non zero. 
\end{Def}
It is known (see \cite{Vire02}) that if $\Ho$ is ribbon and finite type
(i.e. $\dim(\Ho_a)<\infty$ for all $a\in G$), then $\Ho$ has a unique up to
a scalar symmetrized integral closely related to the right integral.
However, the finiteness is not a necessary
conditions as we shall see in Section \ref{s:ex}.   
\subsection{The universal invariant of $G$-links}\label{ss:univ-inv}
Let $\Ho=\{\Ho_{a}\}_{a\in G}$ be a ribbon Hopf $G$-coalgebra and let
$\text{H\!H}_0(\Ho_{a})=\Ho_{a}/[\Ho_{a},\Ho_{a}]$ where $[\Ho_{a},\Ho_{a}]$ is the subspace of $\Ho_{a}$ spanned by $xy-yx$ for $x,y\in \Ho_{a}$. 
\begin{Def}
  A $G$-link is a couple $(L,\omega)$ where $L$ is an oriented framed
  link embedded in $\R^3$ and $\omega$ is a map from the set of
  components of $L$ to $G$.  To reduce some technicalities, we will consider
  $G$-links with ordered components.
\end{Def}

  Let $L$ be  a $G$-link and denote the color of its $i^{th}$-component by $a_i\in G$.  
The universal invariant of $L$ is an element
$J(L,\omega)\in\bigotimes_i \text{H\!H}_0(\Ho_{a_i})$ obtained as follows.  
Let $D$ be a regular planar diagram of $L$ where each component has a marked point.     For each crossing, cup and cap put beads colored with elements of the algebra determined by the following diagrams:
\begin{equation*}
\fbox{$
  \begin{array}{c}\\
\quad
\epsh{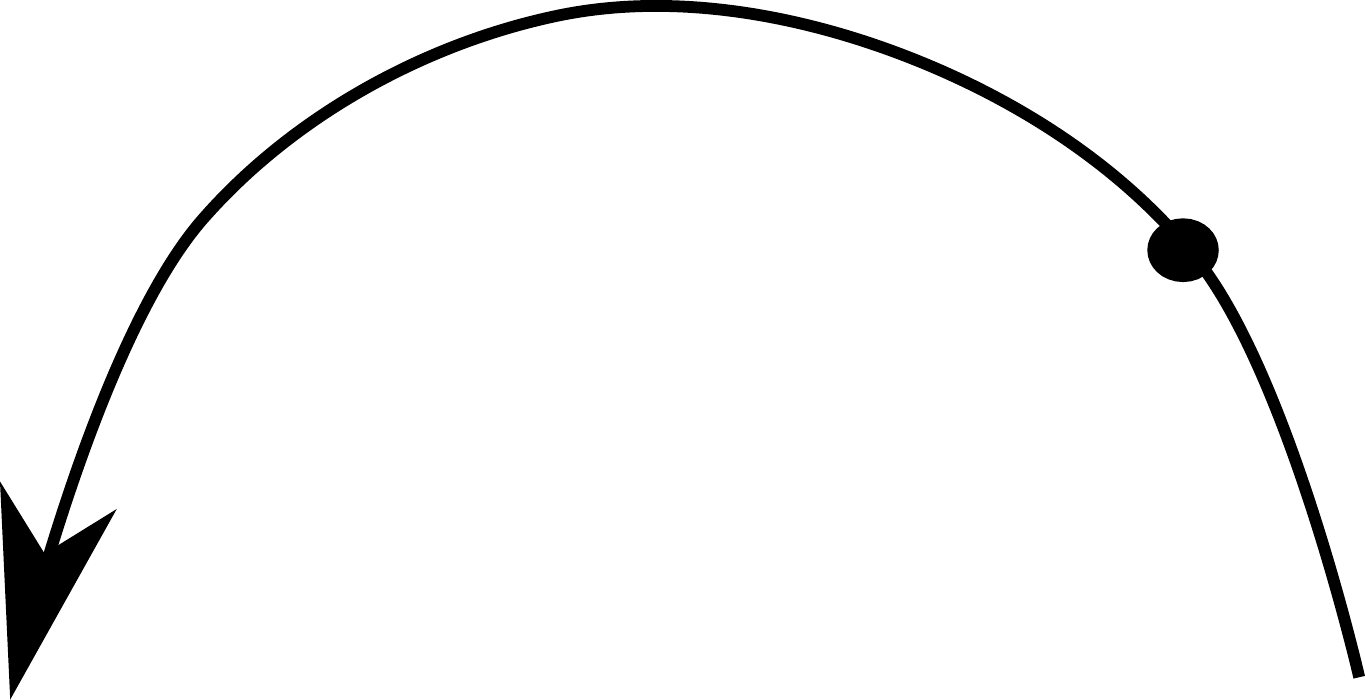}{4ex}
\put(-16,1){\ms{1_{a_i}}} \qquad \epsh{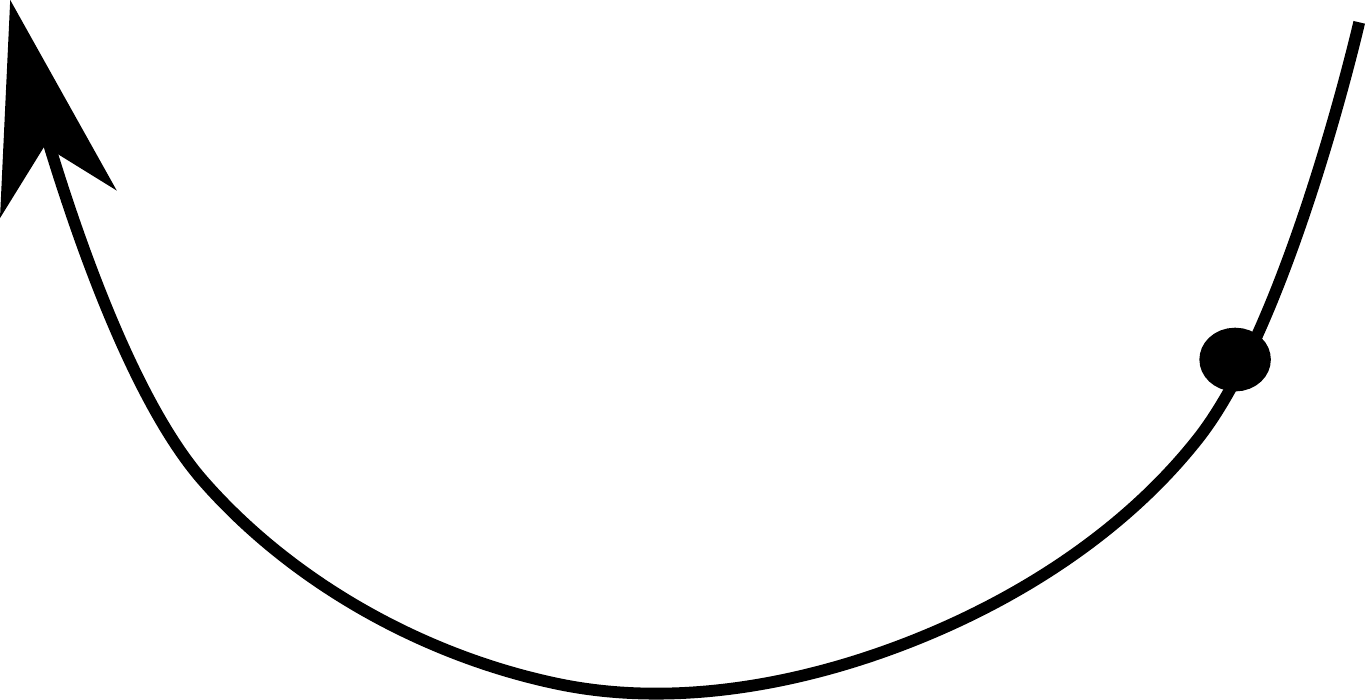}{4ex} \put(-17,3){\ms{1_{a_i}}}
\qquad \epsh{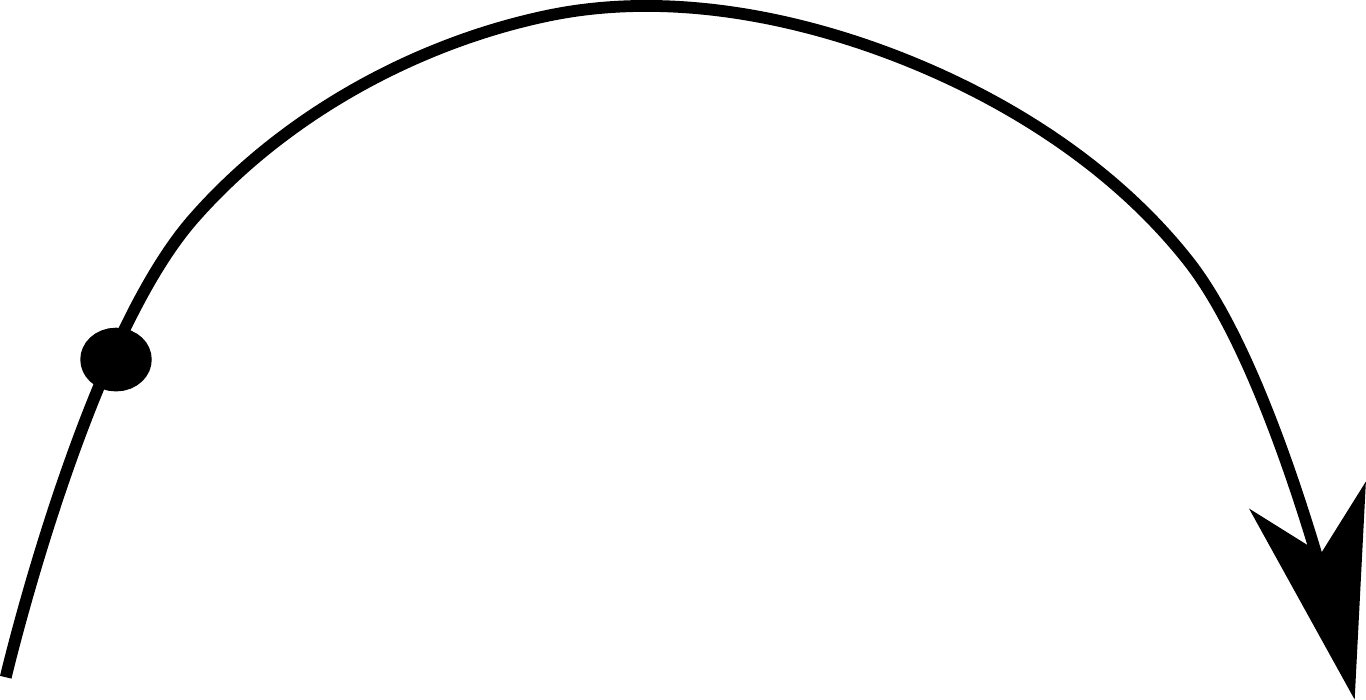}{4ex} \put(-35,0){\ms{g_{a_i}}} \qquad
\epsh{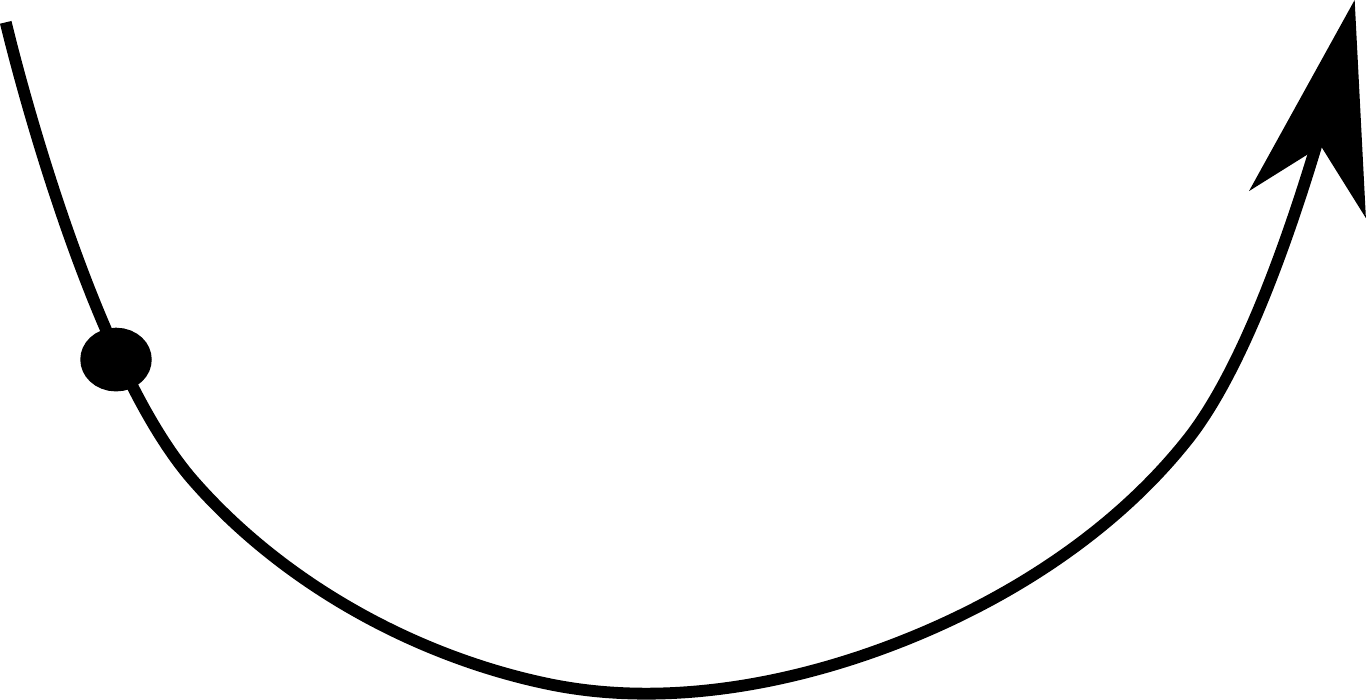}{4ex} \put(-33,2){\ms{{g^{-1}_{a_i}}}}\ \\\\
    \qquad \qquad 
    \epsh{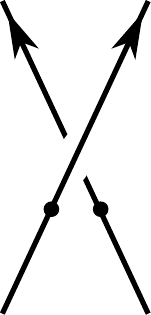}{8ex} \put(-28,-3){\ms{r_{a_i}}} \put(-3,-3){\ms{s_{a_j}}}
    \qquad \qquad \qquad
    \epsh{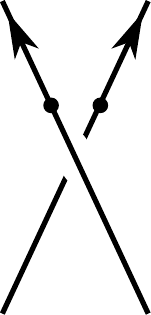}{8ex} \put(-63,11){\ms{S_{-a_j}(r_{-a_j}})}\put(-3,10){\ms{s_{a_i}}}
    \qquad \qquad
    \epsh{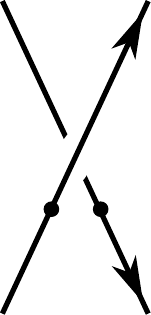}{8ex} \put(-28,-3){\ms{r_{a_i}}} \put(-3,-3){\ms{S_{-a_j}(s_{-a_j})}}
    \qquad \qquad \qquad
    \epsh{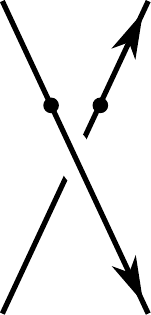}{8ex} \put(-28,11){\ms{r_{a_j}}}\put(-3,10){\ms{s_{a_i}}} 
    \qquad \qquad
    \\\\
 \end{array}$}
\end{equation*}
where the $i$ and $j$ refer to the ordering of the components of $L$
and as above we write $r_{a_i}\otimes s_{a_j}$ for $\RR_{a_i,a_j}$.
The four remaining possible orientations at a crossing are obtained by
reversing simultaneously the orientation of the two strands in the
figure above.
For each component of $D$ (starting at the marked point) we multiply these colored beads with the following rules:
\begin{equation}
  \label{eq:mult-beads}
  \epsh{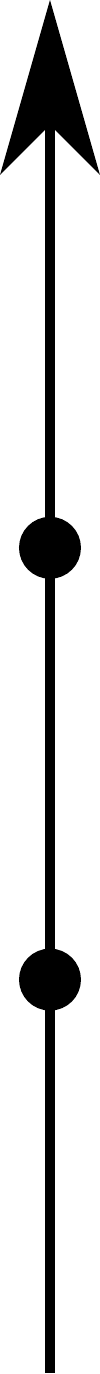}{10ex}\put(-12,8){\ms x}\put(-12,-9){\ms y}\quad
  \equiv\quad\hspace{1ex}\epsh{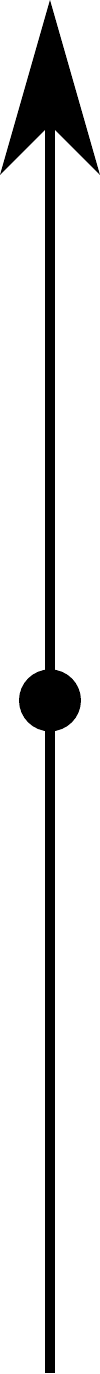}{10ex}\put(-15,2){\ms {xy}}\qquad\text{ and }\qquad
  \epsh{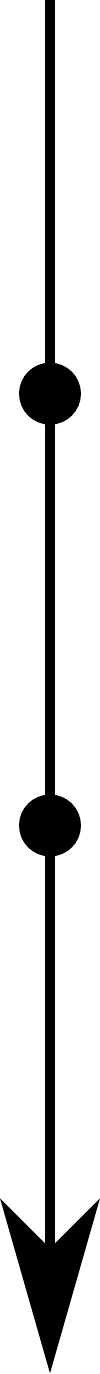}{10ex}\put(-12,12){\ms x}\put(-12,-5){\ms y}\quad
  \equiv\quad\hspace{1ex}\epsh{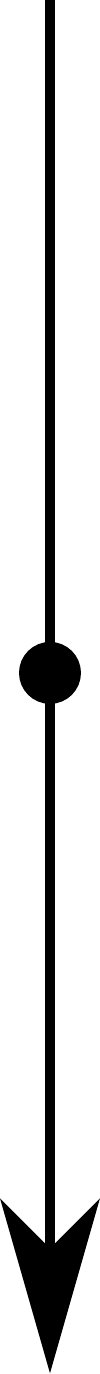}{10ex}\put(-15,2){\ms {yx}}\quad.
\end{equation}
Also beads can freely move around cap and cup with any orientation
\begin{equation}
  \label{eq:cap-beads}
  \epsh{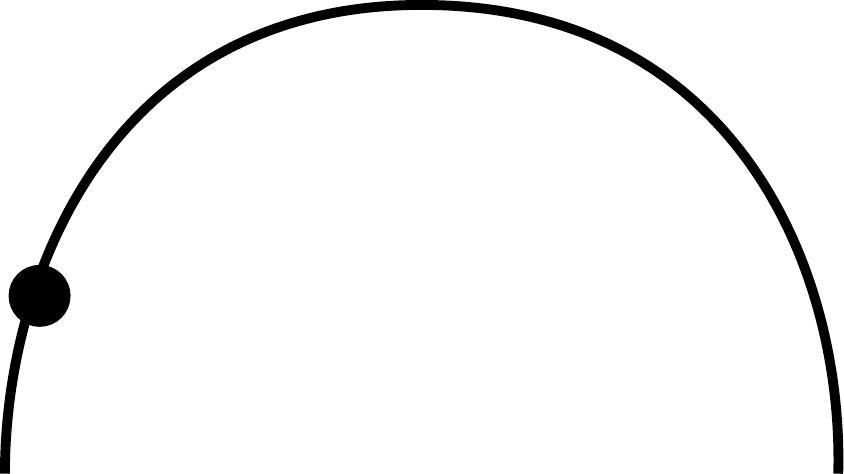}{4ex}\put(-30,-1){\ms x}\quad\equiv\quad
  \epsh{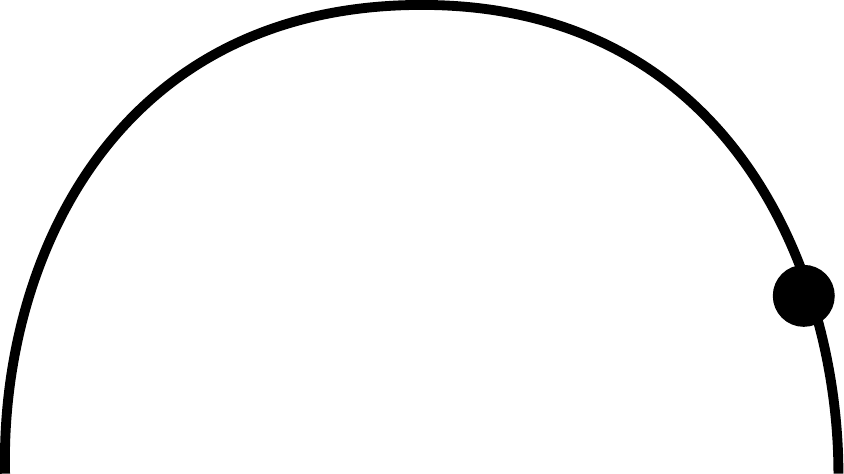}{4ex}\put(-10,-1){\ms {x}}\qquad\text{ and }\qquad
  \epsh{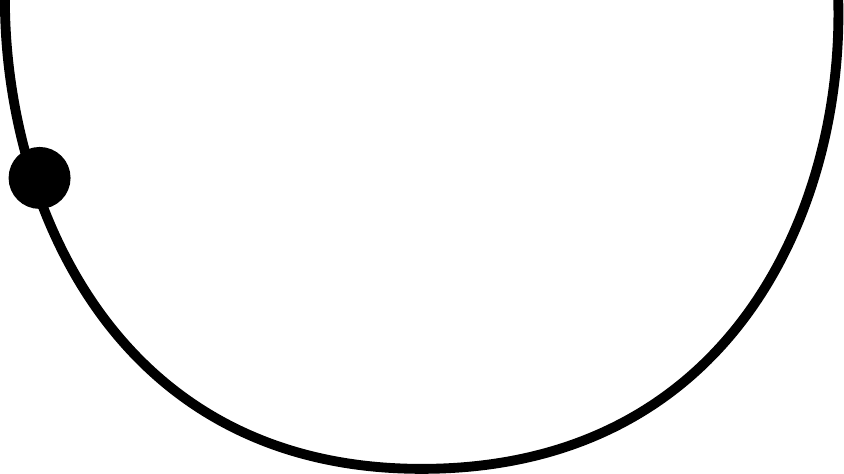}{4ex}\put(-30,4){\ms x}\quad\equiv\quad
  \epsh{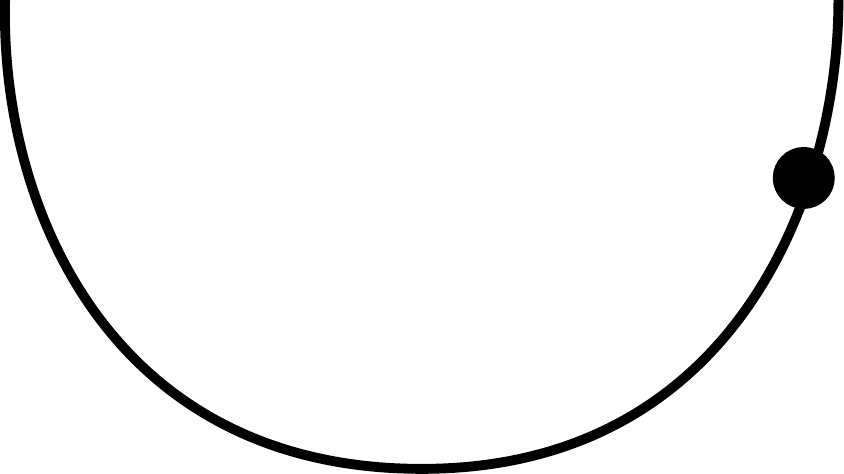}{4ex}\put(-10,4){\ms {x}}\quad.
\end{equation}
and it can pass through a crossing.  For the $i^{th}$ component, from this process we obtain an element in $\Ho_{a_i}$ and consider its image in $\text{H\!H}_0(\Ho_{a_i})=\Ho_{a_i}/[\Ho_{a_i},\Ho_{a_i}]$.  We define the collection of these images as  $J(D,\omega)\in \bigotimes_{i=1}^n \text{H\!H}_0(\Ho_{a_i})$ where $n$ is the number of components of $L$.  
The following theorem is a graded version of Lawrence-Ohtsuki universal
invariant (\cite{RJLa89,Oh93}).  
\begin{theo}\label{T:univ-inv}
  The element
  $J(D,\omega)$ is a 
  diffeomorphism invariant of the $G$-link $(L,\omega)$.  We call this invariant the \emph{universal $\Ho$ invariant of $L$} and denote it by  $J(L,\omega)$.\end{theo}

\begin{proof}
First, because we quotient by the subspace $[\Ho_{a_i},\Ho_{a_i}]$ for each $i$, the element $J(D,\omega)$ does not depend on the choice of the marked points on each component.  
We need to check that if $D$ and $D'$ are two regular planar diagram of $L$ then  $J(D,\omega)=J(D',\omega)$.  It is well known that $D$ and $D'$ are related by isotopy in $\R^2$ and a finite sequence of Reidemeister II and III moves and  the move:
\begin{equation}\label{E:TureavDig1}
\epsh{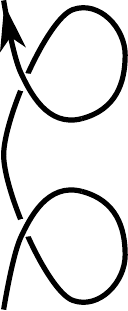}{8ex}\quad\longleftrightarrow\quad\epsh{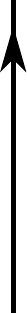}{8ex}
\end{equation}
 Thus, to prove the theorem we need to show these moves can be translated into identities satisfied by the ribbon Hopf $G$-coalgebra $\Ho$.   To show the assignment is invariant under isotopies in $\R^2$ it is enough to show it satisfies the move:
 \begin{equation*}\label{E:TureavDig2}
\epsh{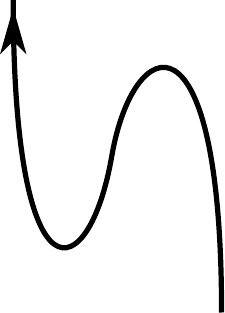}{8ex}\ \longleftrightarrow\ \epsh{fig6b}{8ex}\ \longleftrightarrow\ \epsh{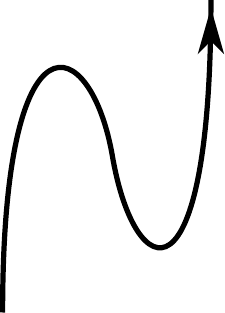}{8ex}\qquad
\epsh{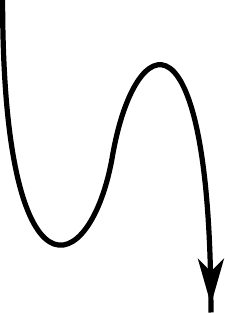}{8ex}\ \longleftrightarrow\ \epsh{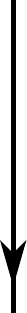}{8ex}\ \longleftrightarrow\ \epsh{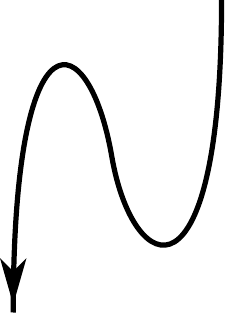}{8ex}
\end{equation*}
This move follows directly from the definition of the universal invariant on cups and caps.

From Proposition 1.4 of \cite{Oh93}, the quasitriangular structure of $\Ho$ implies 
\begin{equation}\label{E:QuStrImplies1}
\RR_{a,b}^{-1}=(1_{a}\otimes S^{-1}_b)\RR_{a,-b}=(S_{-a}\otimes 1_b)\RR_{-a,b},
\end{equation}
\begin{equation}\label{E:QuStrImplies2}
\bp{\RR_{a,b}}_{12c}\bp{\RR_{a,c}}_{1b3}\bp{\RR_{b,c}}_{a23}=
\bp{\RR_{b,c}}_{a23}\bp{\RR_{a,c}}_{1b3}\bp{\RR_{a,b}}_{12c},
\end{equation}
\begin{equation}\label{E:QuStrImplies3}
(S_a \otimes S_b)\RR_{a,b}=\RR_{-a,-b}.
\end{equation}
These relations imply that the Reidemeister II and III moves with all orientations of strands hold.  

Finally, we need to prove Move \eqref{E:TureavDig1}.  Since $\Ho$ is
ribbon then the condition in Definition \ref{D:RibbonStructure}
implies that the universal invariant satisfies the move:
\begin{equation*}
\epsh{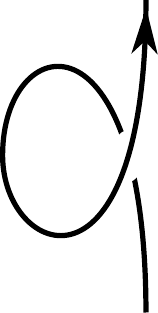}{8ex}\ \longleftrightarrow\ \epsh{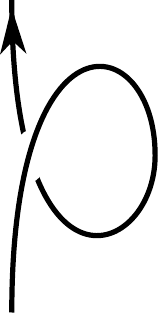}{8ex}
\end{equation*}
After applying this move to the first loop in the first diagram of
Equation \eqref{E:TureavDig1} we can apply Reidemeister II and III
moves to unknot the diagram.
\end{proof}

We now give two well known properties of the universal invariant, analogous propositions have been proved in related situations, see \cite{Vire01} and \cite{GHP20}.  
\begin{Prop}\label{P:antipode}
  Let $(L,\omega)$ and $(L',\omega')$ be $G$-links such that $(L',\omega')$ is be obtained from
  $(L,\omega)$ by reversing the orientation of the $i$\textsuperscript{th}
  component and changing the corresponding value of $\omega$ to its
  opposite.  Then $J(L',\omega')$ is obtained from $J(L,\omega)$ by
  applying the antipode to the $i$\textsuperscript{th} factor.
\end{Prop}
\begin{proof}
 This proposition was proved in a special case of certain  unrolled quantum groups in      \cite[Lemma 5.4]{GHP20}.  The proof of \cite[Lemma 5.4]{GHP20} works word for word in the context of this paper after adding the grading on the antipode map, pivotal element and the R-matrix. 
\end{proof}

\begin{Prop}\label{P:doubling}
  Let $(L,\omega)$ be a $G$-link and $(L',\omega')$ be obtained from
  $L$ by replacing the $i$\textsuperscript{th} component colored by
  $a+b\in G$ by two parallel copies colored by $a$ and $b$
  respectively. Then $J(L',\omega')$ is obtained from $J(L,\omega)$ by
  applying the coproduct $\Delta_{a,b}$ to the $i$\textsuperscript{th}
  factor.
\end{Prop}
\begin{proof}
To prove this proposition it is enough to check on elementary diagrams consisting of a cup, cap or crossing.    For a crossing the relations \eqref{E:DeltaR12} and \eqref{E:DeltaR23} of Definition \ref{Def:quasitriangular} imply the desired equality of beads.  For a cup of cap the equality on beads follows from the fact that the pivotal element is grouplike.
\end{proof}

\subsection{Invariant of $G$-manifolds} From now on we assume  $\K$ is an algebraic closed field. 
A \emph{$G$-manifold} is a pair $(M,\omega)$ where $M$ is an oriented closed $3$-manifold and $\omega\in H^1(M,G)$.
A surgery presentation of a $G$-manifold $M$ is a $G$-link $L$ in
$S^3$ such that $S^3_L$ is homeomorphic to $M$.  Note, we put an ordering on the components of $L$.  Such a link is $G$-colored where the color of the $i^{th}$ 
component of $L$ is defined by the value of $\omega$ on its oriented
meridian in $S^3\setminus L\subset M$ and denoted by $a_i$.

In \cite{Vire01}, Virelizier proves the following theorem, we give a sketch of the proof because similar ideas will be used below when we re-normalize this invariant.
Let $\Ho$ be a ribbon Hopf $G$-coalgebra with a twist non-degenerate symmetrized integral (we assume $\mu_0(g_0^{-1}\theta_0^{-1})^{-1}=\mu_0(g_0\theta_0)=\delta$ for some non zero element $\delta\in \K$, which is always possible, up to rescaling, since $\K$ has square roots).  
\begin{theo}\label{T:InvGman}
For each $G$-manifold
$(M,\omega)$ the assignment 
$$\HV(M,\omega)=\delta^{-s}\bigotimes_{i=1}^n \mu_{a_i}\bp{J(L,\omega)}\in \K, $$
is an invariant of $G$-manifolds 
where $s\in\Z$ is the signature of the linking matrix of $L$. 
\end{theo}
\begin{proof}
Any two presentations of $M$ are
  related by isotopy, orientation reversal of components of $L$ and
  colored Kirby I and II moves which are a stabilization with $\pm1$ framed $0$-colored unknot and sliding along a component of $L$, respectively.  The colored Kirby II move is represented by
  $$\epsh{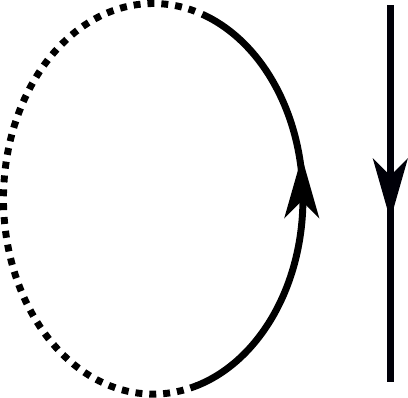}{12ex}\put(-38,4){\ms{a+b}}\put(0,0){\ms{b}}
  \quad\leftrightarrow\quad
  \epsh{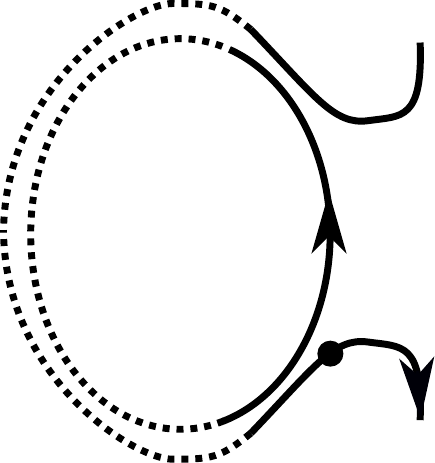}{12ex}\put(-22,4){\ms{a}}
  \put(2,-20){\ms{b}}\put(-13,-18){\ms{g_b}}\;\;\;.
  $$
  To prove the theorem we will show $\HV$ is unchanged under these operations.   
  
   Since
  $\si\circ S=\si$ then Proposition \ref{P:antipode} implies $\HV$ is invariant by orientation reversal of
  a component of $L$.  A Kirby I move increases the signature of $L$ by $\pm1$ and also multiplies the invariant by $\mu_{0}((g_0\theta_0)^{\pm1})=\delta^{\pm1}$.  Thus, $\HV$ is
  globally unchanged under a Kirby I move.  
  Finally, invariance by a colored Kirby II move
  follow from Proposition \ref{P:doubling} and the first property of
  the symmetrized integral.   Remark that the invariance under this move only uses that the
  components colored with $a+b$ and $a$ on the picture and are evaluated with
  the symmetrized integral.
\end{proof}
We call $\HV$ the \emph{graded Hennings-Virelizier invariant}.   When the grading is trivial (i.e.\ $G=0$) then $\HV$ is the Hennings invariant defined in \cite{Henning96}.
As we will see in Section \ref{s:ex} there is a Ribbon Hopf $\C/2\Z$-coalgebra $\U$, associated to $\slt$,  with a twist non-degenerate symmetrized integral.  The degree zero restriction of the invariant $\HV$ associated to $\U$ is the Hennings invariant associated to usual quantum group of $\slt$.  However, in this example $\HV$ vanishes in non-integral degree.  
Next we will discuss how to modify the integral so that one can produce a non-zero invariant where it previously vanished.  

\section{The modified symmetrized integral and modified invariant}\label{S:ModifiedIntInv}
\subsection{Ambidexterity of the symmetrized integral}
Let $\Ho$ be a finite type pivotal Hopf $G$-coalgebra with a
symmetrized integral $\si$.

For a linear endomorphism $f$ of a finite dimensional $\K$-vector
space $V\otimes W$ we denote by $\ptr^\K_W(f)\in\End_\K(V)$ the right
partial trace of $f$ given by
$\ptr^\K_W(f)(v)=\sum_i(\Id_V\otimes w_i^*)(f(v\otimes w_i))$ for any
basis $\{w_i\}$ of $W$ with dual basis $\{w_i^*\}$.  We define similarly
the left partial trace of $f$ denoted $\ptr^\K_V(f)\in\End_\K(W)$.   Remark that denoting by $\tr^\K$ the usual linear trace, we have $\tr^\K_V(\ptr^\K_W(f))=\tr^\K_W(\ptr^\K_V(f))=\tr^\K_{V\otimes W}(f)\in\K$.

Let $L_x: \Ho_a \to \Ho_a$ be the left multiplication by $x\in \Ho_a$.

\begin{Prop}
  Let $a,b\in G$ and 
  $f\in\End_\K(\Ho_{a}\otimes \Ho_{b})$ 
  be a
  morphism of $\Ho_{a+b}$-module (i.e.
  $f(\Delta_{a,b}(x).y\otimes z)=\Delta_{a,b}(x)f(y\otimes
  z)$). Then
  \begin{equation}
    \label{eq:ambi}
    \si_{a}(\ptr^\K_{\Ho_{b}}\bp{(\Id\otimes L_{g_{b}})f}(1_{a}))=
    \si_{b}(\ptr^\K_{\Ho_{a}}\bp{(L_{g_{a}^{-1}}\otimes \Id)f}(1_{b}))
  \end{equation}
\end{Prop}
\begin{proof}\label{P:si-ambi}
  The proof is based on the theory of modified trace: The guiding
  principle is that by \cite{BBG,Ha18} the symmetrized integral is the
  modified trace and we transpose the partial trace property of
  modified trace to a property for the symmetrized integral.

  In this proof we use the notation and result of
  \cite{Ha18}.  Let $\cat$ be the pivotal category of left $\Ho$-modules 
  (an object of $\cat$ is a finite dimensional $\K$-vector space
  equipped with a structure of left $\Ho_a$-module for some $a\in
  G$).
  For $a\in G$, the action of left multiplication on $\Ho_a$ makes $\Ho_a$ into  a projective object of  $\cat$.
  A \emph{cyclic trace} on
  the subcategory  of projective object $\Proj$ of $\cat$ is a family
  of linear function $\bcr{\mt_P:\End_\cat(P)\to\K}_{P\in\Proj}$
  satisfying $\mt_P(fg)=\mt_Q(gf)$ for any $f:Q\to P$ and $g:Q\to
  P$. By \cite[Proposition 2.4]{BBG} 
  and Definition \ref{D:si} (2),
  there is a unique cyclic trace $\mt$ such that for any $a\in G$ and
  $f\in\End_{\cat}(\Ho_a)$ we have $\mt_{\Ho_a}(f)=\si_a(f(1_a))$.
  Definition \ref{D:si} (1) implies that $\si$ is a right symmetrized
  $G$-integral, thus by \cite[Theorem 1.1]{Ha18}, $\mt$ is a right
  modified trace thus satisfies the right partial trace property which
  implies that for any $f\in\End_\cat(\Ho_{a}\otimes\Ho_{b})$,
  $$\mt_{\Ho_{a}\otimes\Ho_b}(f)= \mt_{\Ho_a}(\ptr^\cat_{\Ho_b}(f))$$
  where $\ptr^\cat_{\Ho_b}$ is the partial trace in $\cat$ which uses the pivotal
 structure of $\cat$.  Now Definition \ref{D:si} (3) implies that $\si$ is
  also a left symmetrized $G$-integral (i.e.
  $(\Ho,g)$ is unibalanced) thus $\mt$ satisfies the left
  partial trace property:
  $\mt_{\Ho_{a}\otimes\Ho_b}(f)=
  \mt_{\Ho_b}(\ptr^\cat_{\Ho_a}(f))$. Since in $\cat$ the left and right
  categorical partial trace are given by the vector space partial
  trace precomposed with the action of $g_a^{-1}$ and $g_b$
  respectively, combining the two equalities gives the desired 
  identity.
\end{proof}
\begin{rmq}
  Even if the grading is trivial and $\Ho$ is an ordinary unibalanced
  unimodular Hopf algebra, we do not know a proof  without using modified traces of Proposition
  \ref{P:si-ambi}  which shows an unexpected result about the integral. 
  \end{rmq}
\subsection{The modified symmetrized integral}
Let
$$\Delta_{a_1,...,a_n}:\Ho_{\sum_{i=1}^n
  a_i}\to\Ho_{a_1}\otimes\cdots\otimes\Ho_{a_n}$$ be the $(n-1)$-fold
coproduct.  For example,
$$\Delta_{a_1,a_2,a_3}=\left(\Delta_{a_1,a_2}\otimes \Id_{\Ho_{a_3}}\right)\circ \Delta_{a_1+a_2,a_3}=
  \bp{ \Id_{\Ho_{a_1}}\otimes \Delta_{a_2,a_3}}\circ
 \Delta_{a_1,a_2+a_3}.$$
 Let $C(\Ho_{a_1}\otimes\cdots\otimes\Ho_{a_n})$ be the subspace of $\otimes_{i=1}^n \Ho_{a_i}$ of all elements which commute with the image of $\Delta_{a_1,...,a_n}$, in other words all $x\in \otimes_{i=1}^n \Ho_{a_i}$ such that $$ x \Delta_{a_1,...,a_n}(y) = \Delta_{a_1,...,a_n}(y)x $$ for all $ y\in \Ho_{\sum_{i=1}^n a_i} $ (note here if $n=1$ then $\Delta_{a_1}=\Id_{\Ho_{a_1}}$). 
\begin{Lem}\label{L:equivariant}
  Recall $L_x: \Ho_a \to \Ho_a$ is the left multiplication by $x\in \Ho_a$. Then
  $$\si_{a_1}L_{g_{a_1}^{-1}}\otimes\Id(C(\Ho_{a_1}\otimes\cdots\otimes\Ho_{a_n}))
  \subset C(\Ho_{a_2}\otimes\cdots\otimes\Ho_{a_n})$$
  $$\text{and }
  \Id\otimes\si_{a_n}L_{g_{a_n}}(C(\Ho_{a_1}\otimes\cdots\otimes\Ho_{a_n}))
  \subset C(\Ho_{a_1}\otimes\cdots\otimes\Ho_{a_{n-1}})$$
\end{Lem}
\begin{proof}
  We prove the first
  inclusion, the second is similar.  In this proof we use implied summation when using  coproducts.  Let $y\in \Ho_{\sum_{i=2}^n a_i}$ we use the following notation: 
$$\otimes_{i=0}^ny_i= \Delta_{-a_1,a_1, a_2,...,a_n}(y)
=\left((\Delta_{-a_1,a_1}\otimes \Id_{\otimes_{i=2}^n\Ho_{a_i}})\Delta_{0,a_2,...,a_n}\right)(y)$$
is an element of $\Ho_{-a_1}\otimes \Ho_{a_1}\otimes (\otimes_{i=2}^n\Ho_{a_i})$.  
Then the defining property of the  antipode imply
  $$1_{\Ho_{a_1}}\otimes (\Delta_{a_2,...,a_n}(y))=S_{-a_1}(y_{0})y_{1}\otimes(\otimes_{i=2}^ny_i)= y_{1}S^{-1}_{a_1}(y_{0})\otimes(\otimes_{i=2}^ny_i).$$
We use this equation in the first and last equalities of the following calculation:   
for $\otimes_{i=1}^nx_i\in C(\otimes_{i=1}^n\Ho_{a_i})$ with $x_i\in \Ho_{a_i}$ we have
 \begin{align*}
 \si_{a_1}(g^{-1}_{a_1}x_1)(\Delta_{a_2,...,a_n}(y)).(\otimes_{i=2}^nx_i)
 &=  \si_{a_1}(g_{a_1}^{-1}S_{-a_1}(y_{0})y_{1}x_1)(\otimes_{i=2}^ny_i).(\otimes_{i=2}^nx_i)\\
&=  \si_{a_1}(g_{a_1}^{-1}S_{-a_1}(y_{0})x_1y_{1})(\otimes_{i=2}^nx_i).(\otimes_{i=2}^ny_i)\\
&  =\si_{a_1}(S_{a_1}^{-1}(y_{0})g_{a_1}^{-1}x_1y_{1})(\otimes_{i=2}^nx_i).(\otimes_{i=2}^ny_i)\\
&  =\si_{a_1}(g_{a_1}^{-1}x_1y_{1}S_{a_1}^{-1}(y_{0}))(\otimes_{i=2}^nx_i).(\otimes_{i=2}^ny_i)\\
 & = \si_{a_1}(g_{a_1}^{-1}x_1)(\otimes_{i=2}^nx_i).(\Delta_{a_2,...,a_n}(y))
  \end{align*}
  where the second equality comes from the fact that $\otimes_{i=1}^nx_i$ is in $C(\otimes_{i=1}^n\Ho_{a_i})$, the third from the property of the pivot and the forth from the cyclic property of the integral.
\end{proof}
\begin{Def}\label{Def m-integral} 
Let $X\subset G$ and set $G'=G\setminus X$. A modified symmetrized integral on $G'$ is a family of $\K$-linear maps $$\si'=\{\si'_{a}:\ C(\Ho_a)\to\K\}_{a\in G'}$$ satisfying for any $a,b\in G'$ 
  $$\si_{a}L_{g_{a}^{-1}}\otimes\si'_{b}
  =\si'_{a}\otimes \si_{b}L_{g_{b}} \text{ on } C(\Ho_a\otimes \Ho_b).$$
\end{Def}

\begin{theo}
\label{T:ModSymIntExist}
  Assume $\K$ is an algebraically closed field.
  Let $\Ho_a$ be finite dimensional and semi-simple for all $a\in
  G'$. Then there exists a family $\qn{z_a\in C(\Ho_a)}_{a\in G'}$ such
  that
  $$\si_a(x)=\tr_{\Ho_a}^{\K}(L_{z_ax}) \text{ for all } x\in \Ho_a. $$
  Furthermore, there is a modified symmetrized integral on $G'$ defined by
  $$\si'_a(z)=\si_a(z_a z)=\tr_{\Ho_a}^{\K}(L_{z_a^2z}) \text{ for all } z\in C(\Ho_a).$$
\end{theo}
\begin{proof}
 Since $\K$ is an algebraically closed field then for each $a\in G'$ the algebra $\Ho_a$ is
  a product of matrix algebra.
Since any linear form on $M_n(\K)$ can be realized by $M \mapsto \tr^\K_{M_n(\K)}(L_{N}L_{M})$ for a unique $N\in M_n(\K)$ we have 
$\si_a$ is equal to $x\mapsto\tr^\K_{\Ho_a}(L_{z_ax})$ for some
  unique $z_a\in\Ho_a$.  Property (2) of Definition \ref{D:si} implies that
  $z_a\in C(\Ho_a)$. 

  For the second assertion we need to show that $\si'_a $ is a modified symmetrized integral on $G'$.   Let $a,b\in G'$ and
  $x=x'\otimes x''\in C(\Ho_a\otimes \Ho_b)$ (summation symbol
  omitted). Then let
  $\tilde x=(z_a\otimes z_b).x\in C(\Ho_a\otimes \Ho_b)$ so that
  $L_{\tilde x}$ is an endomorphism of the $\Ho_{a+b}$-module
  $\Ho_a\otimes \Ho_b$ as in Proposition \ref{P:si-ambi}.  Applying
  the proposition, we get following equivalences of equalities:
    \begin{align*}
      \si_{b}(\ptr^\K_{\Ho_{a}}\bp{(L_{g_{a}^{-1}}\otimes \Id)L_{\tilde x}}(1_{b}))
      &=
        \si_{a}(\ptr^\K_{\Ho_{b}}\bp{(\Id\otimes L_{g_{b}})L_{\tilde x}}(1_{a}))
        \\
     \si_b(\ptr^\K_{\Ho_{a}}(L_{g_{a}^{-1}z_ax'})z_bx'')
          &  =\si_a(z_ax'\ptr^\K_{\Ho_{b}}(L_{g_{b}z_bx''}))\\
      \tr^\K_{\Ho_{a}}(L_{g_{a}^{-1}z_ax'})\si_b(z_bx'')
           & =\si_a(z_ax')\tr^\K_{\Ho_{b}}(L_{g_{b}z_bx''})\\
      \si_{a}L_{g_{a}^{-1}}\otimes\si'_{b}(x)
      &=\si'_{a}\otimes \si_{b}L_{g_{b}}(x)
    \end{align*}
    and the proposition follows.  
\end{proof}
\begin{rmq}\label{R:z_a}
  The central elements $z_a$ are determined via a basis $\qn{z_i}$ of
the center $C(\Ho_a)$ in which $z_i$ is the element corresponding to
the identity matrix of $i\ith$-factor in the decomposition
$\Ho_a \cong \prod_i \End(V_i)$, $V_i \simeq \K^{n_i}$ are irreducible
representations of $\Ho_a$. Remark that
$\tr^\K_{\Ho_{b}}(L_{z_i})=\dim_\K(\End(V_i))=n_i^2$.  In particular, if $\K$ has
characteristic $0$ then $z_a=\sum_i\frac{d_i}{n_i} z_i$ where $d_i$ is an
element of $\K$ (note $d_i$ can be interpreted as the modified dimension of
$V_i$, see \cite{Ha18}). 
Note $\si_a(z_i)=n_id_i$ then $\si'_a(z_i)=d_i^2$ and
$\si'_a(1)=\si_a(z_a)=\sum_i d_i^2$.
\end{rmq}

\subsection{The modified invariant} 
For this section let $\Ho$ be a ribbon Hopf $G$-coalgebra with a
twist non-degenerate symmetrized integral and a modified symmetrized
integral $\mu'=(\mu'_a)_{a\in G'}$.

An {\em admissible} $G$-manifold is a $G$-manifold with the requirement
that the cohomology class $\omega\in H^1(M,G)\simeq\Hom(H_1(M),G)$ has
a value in  $G'$ i.e. $\Im( \omega) \cap G'\neq \varnothing$.  A surgery
presentation $L$ of a $G$-manifold is called {\em computable} if there
exists 
a component $L_i$ of $L$ such that $\omega(m_i)\in G'$ where $m_i$ is a meridian of $L_i$. 
\begin{Lem}
Let $M$ be an admissible $G$-manifold. There exists a computable surgery presentation of $M$.
\end{Lem}
\begin{proof}
  Consider any link presentation $L$ of $M$. Since $\Im(\omega)\subset G$ is
  generated by its values on the meridians of $L$, there exists a sequence
  $a_1,\ldots a_k\in G$ such that $\sum_{j=1}^k a_j\in G'$ with $a_j=\omega(m_{i_j})$ where $m_{i}$ is a
  meridian of the $i^{th}$-component  of $L$.  
  We will show that a series of Kirby moves can be applied to $L$ to obtain the desired computable presentation as follows.  Isotope $L$ so that all the strands with the colors $ a_{j}$ are in a small ball as in the left side of the diagram in Equation \eqref{Fig:Stab}.  Now do two Kirby I moves to create both a $+1$ and $-1$ framed $0$-colored unknot.  Sliding one of these unknots over the other we obtain a Hopf link where one of the components is zero framed.  Sliding the strands colored with the $ a_{j}$'s over this zero framed component we obtain the right side of Equation \eqref{Fig:Stab}: 
   \begin{equation}\label{Fig:Stab}
    \epsh{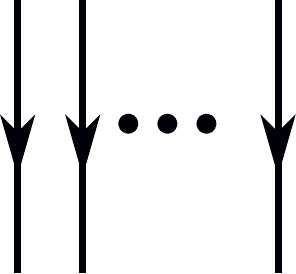}{10ex}\put(-62,-2){\ms{a_1}}\put(-50,-2){\ms{a_2}}\put(-13,-2){\ms{a_k}}
  \longrightarrow
  \epsh{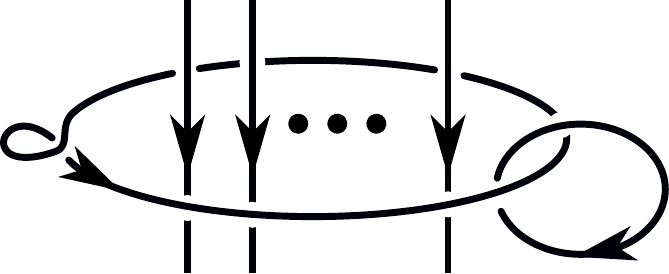}{10ex}\put(-120,-10){\ms{0}}\put(-6,-25){\ms{\sum_ja_j}}.
  \end{equation}
\end{proof}
 
Let $L$ be a $G$-link with $n$ components.  A $k$-string $G$-link $T$
with $\ell$ closed component is a $(k,k)$-pure framed oriented tangle
with $k+\ell$ ordered components colored by elements of
$G$. Its closure $\hat T$ is a $G$-link in $S^3$ with $k+\ell$
components obtained by taking the braid closure (which is well defined
up to isotopy).  We define as in Section \ref{ss:univ-inv} a
universal invariant $J(T)$ by multiplying the beads along components.  The proof of the  following proposition follows in the same way as the proof of Theorem \ref{T:univ-inv}.
\begin{Prop} Let $T$ as above, let $H^T_i=\Ho_{a_i}$ if the i\textsuperscript{th}
  component of $T$ is open and $H^T_i=\text{H\!H}_0(\Ho_{a_i})$ if the
  i\textsuperscript{th} component of $T$ is closed.  The element
  $J(T)\in\bigotimes_{i=1}^{n}{H^T_i}$ is an isotopy invariant of $T$.
  Furthermore, let $cl:\Ho_{a_i}\to\text{H\!H}_0(\Ho_{a_i})$ be the map
  that send $x$ to the class of $g_{a_i}x$ then we have
  $$cl^{\otimes k}(J(T))=J(\hat T)$$
\end{Prop}

Suppose $(M,\omega)$ is an $G$-manifold with surgery
presentation $L$. Let $T_j$ be a $1$-string $G$-link whose closure is $L$.  Then
the universal invariant of $T_j$ is given by
$$ J\bp{T_j,\omega}\in \bigotimes_{i=1}^{j-1}\text{H\!H}_0(\Ho_{a_i})\otimes
\Ho_{a_j}\otimes \bigotimes_{i=j+1}^n \text{H\!H}_0(\Ho_{a_i}).$$
\begin{Lem}\label{central}
  We have
  $$\bp{\bigotimes_{i=1}^{j-1}\si_{a_i}\otimes \Id_{\Ho_{a_j}}\otimes
  \bigotimes_{i=j+1}^n\si_{a_i}}\bp{J\bp{T_j,\omega}}\in
C(\Ho_{a_j}).$$
\end{Lem}
\begin{proof}
  Consider a $n$-string $G$-link $T$ whose braid closure of the $n-1$
  right-most strands produces $T_j$.  Then one can prove (see for
  example \cite[Theorem 5.1]{GHP20}) that
  $J(T)\in C(\bigotimes_{i=1}^{n}\Ho_{a_i})$.  Then applying $n-1$
  times Lemma \ref{L:equivariant} we have the above property.
\end{proof}
\begin{rmq}
  Applying $\si_{a_i}$ to $g_{a_i}$ times the element of the previous
  lemma produces $\delta^s\HV(M,\omega)$.  Suppose now that for
  $a\in G'$, $H_a$ is semi-simple and all its irreducible
  representations have zero quantum dimension (where the quantum
  dimension is defined as the trace on the action of $g_a$).  Then it
  follows that $g_aC(H_a)$ is in the kernel of $\si_a$.  In this
  situation, we get that $\HV(M,\omega)=0$ for any admissible
  $G$-manifold $(M,\omega)$.  This motivate the renormalized invariant
  of Theorem \ref{T:HV'}.
\end{rmq}

Now let $(M,\omega)$ be an admissible $G$-manifold with computable
surgery presentation $L$.  Suppose the coloring of the $j^{th}$
component of $L$ is in $G'$ then set
$$\HV'\bp{M,\omega}= \delta^{-s}\bigotimes_{i=1}^{j-1}\si_{a_i}\otimes
\si'_{a_j}\otimes \bigotimes_{i=j+1}^n\si_{a_i}\bp{J\bp{T_j,\omega}}.$$
\begin{theo}\label{T:HV'}
  $\HV'\bp{M,\omega}$ is an invariant of the admissible $G$-manifold $(M,\omega)$.
\end{theo}
\begin{proof}
  Given a computable $G$-link $L$ with one marked point on one
  component colored by an element of $G'$, consider the complement of
  a small 3-ball centered at the marked point.  It is diffeomorphic to
  a $1$-string $G$-link whose closure is $L$.  Let denote by $\HV'(L)$
  the image of this $1$-string $G$-link by the map of Lemma
  \ref{central}.  The scalar $\HV'(L)$ only depends of the isotopy
  class of $L$.  Property of $\mu'$ implies that moving the marked
  point from one $G'$-colored component to an other does not change
  $\HV'(L)$.  The proof of Theorem \ref{T:InvGman} implies the scalar
  $\HV'(L)$ is also invariant by Kirby I moves between computable
  links.  Moreover, it is also invariant by any Kirby II moves when
  one slides any strand of a component that does not have the marked
  point.  We call such Kirby II moves admissible.
  
  Let $L^0,L^N$ be two computable $G$-link with marked points $m_0$ and $m_N$, respectively.  Then $L^0$ and $L^N$ are related by a sequence of
  ambient isotopies and Kirby I and II moves which correspond to a sequence of
  $G$-links $L^1,L^2,\cdots,L^{N-1}$.  Here it is possible that these Kirby moves are not admissible (two problems can happen: no $G'$ color or one of the moves could slide over the strand with the marked point).  However, we will show this sequence can be used to  produce a new sequence that only contains admissible Kirby
  moves.   For this we perform a series of admissible Kirby moves similar to the moves depicted  in Equation \eqref{Fig:Stab}:  close to a strand of $L^0$ colored with an element of  $G'$  do two Kirby I moves and a Kirby II move to create a Hopf link with a zero and a $+1$ framed component.  Slide the $G'$ colored strand of $L^0$ over the zero framed component to create an  link with a new $G'$ colored unknot (as in the right side of Equation \eqref{Fig:Stab}).  Move the marked point $m_0$ to this unknot to create a new admissible link presentation contains $L^0$ as a sublink.  Now we can  
   perform the
  same Kirby moves ignoring the newly created Hopf link to get a
  sequence of admissible Kirby moves
  $\wb{L^0}\to\wb{L^1},\cdots,\wb{L^{N}}$.   By construction $\wb{L^N}$ is the link $L^N$ with a Hopf link stabilization as depicted in the right side of Equation \eqref{Fig:Stab} with a marked point on the zero framed component of the Hopf link.  Move this marked point to the marked point $m_N$ on the sublink $L^N$.
  Preform the reverse admissible Kirby moves to remove the Hopf link of $\wb{L^N}$ and obtain $L^N$.  
  Since all moves are admissible we have that $\HV'(L^0)=\HV'(L^N)$.
\end{proof}
\section{Examples from quantum groups}\label{s:ex}
Here we show the unrolled quantum group associated to $\slt$ gives
rise to ribbon Hopf coalgebra with a modified symmetrized
integral.  
We first describe a Hopf algebra that embed into the topological
unrolled quantum group $\wh{U_\xi^H}\slt$ of \cite{GHP20}.  \\
Let $\ell\geq 3$ be an integer, let $\ell'=\ell/\gcd(\ell,2)$
and
$\xi=\e^{\frac{2\pi i}{\ell}}$. We use the notations
$\xi^x=\exp(2i\pi x/\ell),\ \qn{x}=\xi^x-\xi^{-x} $ and
$\qn{n}!=\qn{n}\qn{n-1}...\qn{1}$.

Let $\UK$ be the $\C$-algebra with generators $E$, $F$,
$K^\alpha$ where $\alpha\in\C$ with relations: for any
$\alpha,\beta\in\C,$
\begin{align}
  K^\alpha K^\beta&=K^{\alpha+\beta}, & K^\alpha EK^{-\alpha}&=\xi^{2\alpha} E, &
   K^\alpha FK^{-\alpha}&=\xi^{-2\alpha}F, \notag\\
  K^0&=1, & [E,F]&=\frac{K-K^{-1}}{\xi-\xi^{-1}}, & E^{\ell'}=&0=F^{\ell'}\notag.
\end{align}
The algebra $\UK$ is a Hopf algebra where the coproduct, counit and
antipode are defined by
\begin{align*}
  \Delta(E)&= 1\otimes E + E\otimes K, 
  &\ve(E)&= 0, 
  &S(E)&=-EK^{-1}, 
  \\
  \Delta(F)&=K^{-1} \otimes F + F\otimes 1,  
  &\ve(F)&=0,& S(F)&=-KF,
    \\
  \Delta(K^\alpha)&=K^\alpha\otimes K^\alpha
  &\ve(K^\alpha)&=1,
  & S(K^\alpha)&=K^{-\alpha}.
\end{align*}
The Hopf algebra $\UK$ is pivotal with pivot $g=K^{1-\ell'}$.  Let
$G=\C/2\Z$.  For ${a}\in G$, set
$\U_{a}=\UK/(K^{\ell/2}-\xi^{\ell a/2})$.  The Hopf algebra
structure of $\UK$ induces a structure of pivotal Hopf $G$-coalgebra
on the family $\U=\qn{\U_{ a}}_{a\in G}$. We denote the
coproduct, counit and antipode of $\U$ by
$\Delta = \qn{\Delta_{ a,  b}}_{ a,  b\in G},\
\varepsilon,$ and $S=\qn{S_{{a}}}_{ a\in G}$, respectively.

In \cite{Oh93}, Ohtsuki proved a version of the following theorem for a colored Hopf algebra closely related to $\U$. 
\begin{theo}\label{ribbon structure} The Hopf $G$-coalgebra $\U$ is ribbon with $R$-matrix
  \begin{align*}
    &\RR_{ a, b}= \HH_{ a, b}\check{\RR}_{ a, b}\in\U_{ a}\otimes\U_{ b},\\
    \text{with }&\check{\RR}_{ a, b}=\sum_{n=0}^{\ell'-1}\frac{\qn{1}^{2n}}
                 {\qn{n}!}\xi^{\frac{n(n-1)}{2}} E^n\otimes F^n\in\U_{ a}\otimes\U_{ b}\\
    \text{and where }&\HH_{a,b}=\frac{1}{\ell'}
    \sum_{m_1, m_2=0}^{\ell'-1} \xi^{-2(m_1+\beta)(m_2+\alpha)}
    K^{\beta+m_1}\otimes K^{\alpha+m_2}\in\U_{ a}\otimes\U_{ b}
  \end{align*}
  does not depends of $\alpha,\beta\in\C$ congruent to $\frac a2,\frac b2$ respectively modulo $\Z$.
  \end{theo}
 \begin{proof}
In this proof we use the  topological ribbon Hopf algebra $\wh{\cU_\xi^H}\slt$
defined in \cite[Examples 2.3 and 2.7 and Section 3.2]{GHP20}.  For $a\in\C/2\Z$, let 
$I_a=(\xi^{\ell H}-\xi^{\ell a/2})\wh{\cU_\xi^H}\slt$ the ideal of $\wh{\cU_\xi^H}\slt$.  The 
assignment $E\mapsto E$, $F \mapsto F$ and
$K^\alpha\mapsto\xi^{2\alpha H}$ extends to an embedding of the Hopf
algebra $\UK$ into $\wh{\cU_\xi^H}\slt$ which induces an
embedding of (topological) Hopf $G$-coalgebra $\{\U_a\to\wh{\cU_\xi^H}\slt/I_a\}_{a\in G}$.  We will use this
embedding to define a ribbon structure on $\U$ as follows.
 
The $R$-matrix element of $\wh{\cU_\xi^H}\slt$ is given by
$$\RR=\xi^{{2}{H\otimes H}}\sum_{n=0}^{\ell'-1}
\frac{\qn{1}^{2n}}{\qn{n}!}\xi^{\frac{n(n-1)}{2}} E^n\otimes F^n\in\wh{\cU_\xi^H}\slt\wh \otimes \,\wh{\cU_\xi^H}\slt.$$
We get the $R$-matrix of the Hopf $G$-coalgebra $\U$ by
applying the discrete Fourier transform in \cite[Proposition
4.6]{Ha18a} as follow:
Let $a, b\in\C/2\Z$ and $\alpha$, $\beta\in\C$ such that
$a\equiv 2\alpha, b\equiv2\beta$ modulo $2\Z$.  The function
$\xi^{2(H-\alpha)\otimes (H-\beta)}$ is $\ell'$-periodic on
$( a, b)=(\alpha,\beta)+\mathbb{Z}^2$ thus it can be expressed as
a polynomial in $K\otimes1=\xi^{2H\otimes1}$ and $1\otimes K=\xi^{1\otimes 2H}$.
$$\mathcal{F}_{( a,  b)}\bp{\xi^{2(H- \alpha)\otimes (H- \beta)}}=\sum_{m_1, m_2=0}^{\ell'-1}c_{m_1m_2}K^{m_1}\otimes K^{m_2},$$
where 
\begin{align*}
  c_{m_1m_2}&=\frac{1}{\ell'^2}
              \sum_{i_1,i_2=0}^{\ell'-1}\xi^{-2m_1(\alpha+i_1)-2m_2(\beta+i_2)}\xi^{2i_1i_2}\\
            &=\frac{1}{\ell'^2}\xi^{-2m_1\alpha-2m_2\beta}
              \sum_{i_1=0}^{\ell'-1}\xi^{-2m_1i_1}\sum_{i_2=0}^{\ell'-1}\xi^{2(i_1-m_2)i_2}\\
            &=\frac{1}{\ell'^2}\xi^{-2m_1\alpha-2m_2\beta}
              \sum_{i_1=0}^{\ell'-1}\xi^{-2m_1i_1}\ell'\delta_{0, i_1-m_2}\\
            &=\frac{1}{\ell'}\xi^{-2m_1\alpha-2m_2\beta}\xi^{-2{m_1m_2}}.
\end{align*}
The third equality holds as $\xi^2$ is a primitive $\ell'$ root of unity.
Let
$$\HH_{a,b}=\frac{1}{\ell'} \sum_{m_1, m_2=0}^{\ell'-1} \xi^{-2(m_1+\beta)(m_2+\alpha)}
K^{\beta+m_1}\otimes K^{\alpha+m_2}\in\U_a\otimes\U_b,$$ then (see
\cite[Proposition 4.3]{GHP20}),
$\xi^{2H\otimes H}=
\HH_{a,b}\in\wh{\cU_\xi^H}\slt/I_a\wh\otimes\,\wh{\cU_\xi^H}\slt/I_b$.
Hence, the image of $\RR$ in
$\wh{\cU_\xi^H}\slt/I_a \wh\otimes\, \wh{\cU_\xi^H}\slt/I_b$ belongs
to $\U_{a}\otimes \U_{b}$ and is equal to $\RR_{a,b}$.

Since $\RR$ is the universal $R$-matrix of the topological ribbon Hopf
algebra $\wh{\cU_\xi^H}$ with pivot $K^{1-\ell'}$, the properties for
$(\RR_{a,b})_{a,b\in G}$ defined in Definition
\ref{Def:quasitriangular}
and Equation
\eqref{E:twist} are direct transposition of the properties of $\RR$.
 \end{proof}

Now we consider a finite type sub Hopf $G$-coalgebra of $\U$: For
$ a\in\C/{2\Z}$, let $\Uo_{ a}$ be the subalgebra of
$\U_{ a}$ generated by $E$,$F$ and $K$.  By the PBW theorem this
algebra is finite with dimension ${\ell'}^3$ and has a basis $\{E^{n_E}F^{n_F}K^{n_K}\}_{0\le n_E,n_F,n_K<\ell'}$. The family
$\Uo=\{\Uo_{ a}\}_{ a\in G}$ forms a finite type unimodular pivotal
sub Hopf $G$-coalgebra of $\U$  (which is not ribbon), see \cite{Ha18}.  Its symmetrized integral $\sio$ is given in the PBW basis by sending all
vectors to $0$ except $\sio_{ a}(E^{\ell'-1}F^{\ell'-1})=\eta$ for
some arbitrary non zero constant independent of $ a\in G$.

By Theorem \ref{T:ModSymIntExist}, there exists a unique
$z_{ a}\in C(\Uo_{ a})$ such that
$\sio_{ a}(x)=\tr^\C_{\U_{ a}}(L_{z_{ a}x})$ and there is a
modified symmetrized integral ${\sio}'$ given for
${ a}\in G\setminus\{\wb0,\wb1\}$ by
${\sio}'_{ a}(x)=\tr^\C_{\U_{ a}}(L_{z_{ a}^2x})=\sio_a(z_{ a}x)$.

\begin{theo}\label{T:USymmInt}
  The pivotal Hopf $G$-coalgebra $\bp{\U,g}$ has a symmetrized integral $\si$ given by
  \begin{equation}
    \label{eq:symint}
    \si_{ a}:\left|
      \begin{array}{l}E^{n_E}F^{n_F}K^{n_K}\mapsto
        0\text{ if }
        \begin{array}[t]{l}
          {n_E}\neq \ell'-1\text{ or }{n_F}\neq \ell'-1\\
        \text{ or }{n_K}\in \C\setminus\frac\ell2\Z,
        \end{array}
        \\
        E^{\ell'-1}F^{\ell'-1}K^{n\frac\ell2}\mapsto \xi^{n\frac\ell2 a}\eta,
      \end{array}\right. 
  \end{equation}
  for an arbitrary non zero constant $\eta$.
  Furthermore $\U$ has a modified symmetrized integral on
  $G'=\C/2\Z\setminus\{\wb0,\wb1\}$ given by $\si'_{ a}=\si_{ a}\circ L_{z_{ a}}$.
\end{theo}
\begin{proof}
  For $ a\in\C/{2\Z}$, let
  $M_{ a}=\bigoplus_{\alpha\in ]0,1[+i\R}K^\alpha\Uo_{ a}$.
  Then $\U_{ a}=\Uo_{ a}\oplus M_{ a}$ with
  $S_{ a}(M_{ a})\subset M_{- a}$ and
  $\Delta_{ a, b}(M_{ a+ b})\subset M_{ a}\otimes M_{ b}$.
  Let $\pi_{ a}:\U_{ a}\to\Uo_{ a}$ be the projection with kernel $M_{ a}$ and
  define $\si_{ a}=\sio_{ a}\circ \pi_{ a}$.  Then for any
  $x=h+m\in\U_{ a+ b}=\Uo_{ a+ b}\oplus M_{ a+ b}$,
  \begin{align*}
    \si_{ a}\otimes g_{ b}(\Delta_{ a, b} x)&=\si_{ a}\otimes g_{ b}(\Delta_{ a, b} h)+\si_{ a}\otimes g_{ b}(\Delta_{ a, b} m)\\
  &=\sio_{ a}\otimes g_{ b}(\Delta_{ a, b} h)=\sio_{ a+ b}(h)1_{ b}=\si_{ a+ b}(x)1_{ b}
  \end{align*}
  where the second equality comes from
  $\pi_{ a}\otimes\Id(\Delta_{ a, b} m)=0$.  This proves point (1) of
  Definition \ref{D:si}.  Point (3) follows for $x=h+m\in\U_{ a}$ from
  $$\si_{- a}(S_{ a}(x))=\sio_{- a}(S_{ a}(h))+\si_{- a}(S_{
    a}(m))=\sio_{ a}(h)+0=\si_{ a}(x).$$  Finally we take advantage of
  the $\Z$-grading in $\U$ by the weights where the weights of $E$,
  $F$ and $K^\alpha$ are respectively $2,-2$ and $0$.  Remark that $\si$
  vanishes on homogeneous elements of non zero weight. Let
  $\alpha\in [0,1[+i\R$ and $\beta\in ]-1,0]+i\R$ and fix
  $x\in K^\alpha\U_{ a}$ and $y\in K^\beta\U_{ a}$ with homogeneous
  weights $w_x$ and $w_y$ respectively. Then
  $\si_{ a}(xy)=\si_{ a}(yx)=0$ unless $\beta=-\alpha$ and $w_y=-w_x$.
  In this last case $K^\alpha$ commutes with $xy$ so using that
  $K^{-\alpha}x,yK^\alpha\in\Uo_{ a}$ we have
  $$\si_{ a}(xy)=\sio_{ a}((K^{-\alpha}x)(yK^\alpha))=\sio_{
    a}((yK^\alpha) (K^{-\alpha}x))=\si_{ a}(yx).$$ 
  
  We now check that the family
  $\qn{\si'_{ a}}_{ a\in G\setminus \qn{\wb0, \wb1}}$ satisfies
  Definition \ref{Def m-integral}.
  Since $z_aM_a\subset M_a$ and $g_a^{\pm1}M_a\subset M_a$, for
  $x\in \U_{ a}, y\in \U_{ b}$ with $x\in M_{ a}$ or $y\in M_{ b}$ then we have
  $\si_{ a}L_{g_{ a}^{-1}}\otimes \si'_{ b}(x\otimes y)=0=\si'_{
    a}\otimes \si_{ b}L_{g_{ b}}(x\otimes y)$.  On the other hand, if
  $x\in \Uo_{ a}, y\in \Uo_{ b}$ then
\begin{align*}\si_{ a}L_{g_{ a}^{-1}}\otimes \si'_{ b}(x\otimes y)
&=\sio_{ a} L_{g_{ a}^{-1}}\otimes \sio'_{ b}(x\otimes y)\\
&=\sio'_{ a}\otimes \sio_{ b}L_{g_{ b}}(x\otimes y)\\
&=\si'_{ a}\otimes \si_{ b}L_{g_{ b}}(x\otimes y).
\end{align*}
\end{proof}

As in the proof of Theorem \ref{ribbon structure}, the discrete Fourier transform on the twist of $\wh{\cU_\xi^H}$
$$\theta=\sum_{n=0}^{\ell'-1}\xi^{\frac{n^2+3n}2}\frac{{\qn1}^{2n}}{\qn n!}\xi^{2H^2}K^{\ell'-1-n}E^nF^n$$
and on its inverse
$$\theta^{-1}=\sum_{n=0}^{\ell'-1}(-1)^n\xi^{-\frac{n^2+3n}2}\frac{{\qn1}^{2n}}{\qn n!}\xi^{-2H^2}K^{n+1-\ell'}E^nF^n$$
gives a formula for the twists $\theta_0^{\pm}$ in $\U_{0}$ such that 
$$\si_0(g\theta_0)=\wb{\si_0(g^{-1}\theta_0^{-1})}=\lambda\frac{(-1)^{\ell+1}\xi^{-1}}{{\ell'}^2}{\qn1}^{2\ell'-2}\sum_{k=0}^{\ell'-1}\xi^{2k^2+2k}$$
where we assume the normalization factor $\lambda$ is in $\R^*$.  Finally, the last Gauss sum in the previous equation is known to vanish if and only if $\ell\in8\Z$. Thus, the Hopf $\C/2\Z$-coalgebra $\U$ lead to an
invariant of $\C/2\Z$-manifolds $\HV$ and a modified invariant $\HV'$
of admissible $\C/2\Z$-manifolds.

We finish with a simple example which shows that $\HV'$ is a non-trivial renormalization of $\HV$.
\begin{Prop} Assume $\ell\in2\Z\setminus8\Z$. 
  Let $M=S^2\times S^1$ and
  $\omega$ be a cohomology class on $M$ which takes value $a\in\C/2\Z$
  on a generator of $H_1(M)$.  Then $\HV(M,\omega)=0$.  However, if $a\notin\{\wb0,\wb1\}$ then
  $$\HV'(M,\omega)=\left\{
    \begin{array}{cl}
      -2\frac{\qn1^{4\ell'-4}}{\qn{\ell'a}^2}\eta^2&\text{if }\ell>4\\
      2\frac{\qn a^2}{\qn{2a}^2}\eta^2&\text{if }\ell=4 \; .
    \end{array} \right.$$ 
\end{Prop}
\begin{proof}
  A surgery presentation of $M$ is given by an unknot colored by
  $a$. The universal invariant is then $g_a=\xi^{-\ell'a}K$ and
  $\si_a$ vanishes on it.  Let now $\alpha\in\C$ be in the class of
  $a\in\C/2\Z\setminus\{\wb0,\wb1\}$.  To compute $\HV'$ we open the
  unknot to obtain an open straight strand.  The universal invariant is then
  $1_a$ and one should compute
 $$\HV'(M,\omega)=\si'_a(1_a)=\si_a(z_a).$$  Now $\Uo_a$ has $\ell'$
  irreducible representations of dimensions $\ell'$:
  $\{V_{\xi^{\alpha+2k}}:k=0\cdots\ell'-1\}$ where $\xi^{\alpha+2k}$
  is the highest weight of $V_{\xi^{\alpha+2k}}$ (i.e. the eigenvalue
  of $K$ on the kernel of $E$ in $V_{\xi^{\alpha+2k}}$). Let us write
  $z_a=\sum_{k=0}^{\ell'-1}\frac{d_{\alpha+2k}}{\ell'}z_{\alpha+2k}$
  as in Remark \ref{R:z_a}. As shown by the second author in
  \cite{Ha18} (see also \cite{BBG}),
  $d_{\alpha+2k}=\tr^\C(L_{z_az_{\alpha+2k}})=\sio_a(z_{\alpha+2k})$
  is (a normalization of) the modified dimension of
  $V_{\xi^{\alpha+2k}}$.  The explicit computation done for an even root
  of unity gives that
  $$d_{\alpha+2k}=d_0\frac{\ell'\qn{\alpha+2k}}{\qn{\ell'\alpha}}
  \text{ where }d_0=\frac{\qn1^{2\ell'-2}}{{\ell'}^3}\eta.$$
  Then for $\ell>4$,
  \begin{align*}
    \si'(1_a)&=\sum_{k=0}^{\ell'-1}d_{\alpha+2k}^2=\sum_{k=0}^{\ell'-1}d_0^2\ell'^2
               \dfrac{\xi^{2(a+2k)}+\xi^{-2(a+2k)}-2}{\qn{\ell'a}^2}\\
             &=\dfrac{d_0^2\ell'^2}{\qn{\ell'a}^2}
               \bp{-2\ell'+\xi^{2a}\sum_{k=0}^{\ell'-1}\xi^{2k}
               +\xi^{-2a}\sum_{k=0}^{\ell'-1}\xi^{-2k}}\\
             &=-2\dfrac{{\ell'}^3d_0^2}{\qn{\ell'a}^2}
               =-2\frac{\qn1^{4\ell'-4}}{{\ell'}^3\qn{\ell'a}^2}\eta^2,
  \end{align*}
  where we used that the last two sum vanish for $\ell>4$.
\end{proof}
The example of this section is the Hopf $G$-coalgebra version of the invariants of \cite{GHP20}.
\bibliographystyle{plain}
\bibliography{Reference}
\end{document}